\documentclass[a4paper,12pt]{jfg_mdfy}
\usepackage{enumerate}
\usepackage[numbers,sort&compress]{natbib}

\address[hino@sigmath.es.osaka-u.ac.jp]{Masanori Hino, Graduate School of Engineering Science, Osaka University, Osaka 560-8531, Japan}

\numberwithin{equation}{section}

\newtheorem{thm}{Theorem}[section]

\newtheorem{lem}[thm]{Lemma}
\newtheorem{prop}[thm]{Proposition}
\newtheorem{conj}[thm]{Conjecture}

\theoremstyle{definition}
\newtheorem{definition}[thm]{Definition}
\newtheorem{example}[thm]{Example}
\newtheorem{remark}[thm]{Remark}

\makeatletter
\def\rom#1{\mbox{\leavevmode\skip@\lastskip\unskip\/%
           \ifdim\skip@=\z@\else\hskip\skip@\fi{\rm{#1}}}}
\makeatother
\newcommand{\Thm}[1]{Theorem~\ref{th:#1}}

\newcommand{\Lem}[1]{Lemma~\ref{lem:#1}}

\newcommand{\Exam}[1]{Example~\ref{ex:#1}}
\newcommand{\Conj}[1]{Conjecture~\ref{conj:#1}}
\newcommand{\Fig}[1]{Figure~\ref{fig:#1}}
\newcommand{\Eq}[1]{\rom{(\ref{eq:#1})}}
\newcommand{\gm}{\gamma}\newcommand{\dl}{\delta}

\newcommand{\lm}{\lambda}\newcommand{\kp}{\kappa}
\newcommand{\sg}{\sigma}

\newcommand{\om}{\omega}
\newcommand{\Sg}{\Sigma}

\newcommand{\D}{{\Bbb D}}
\newcommand{\N}{{\Bbb N}}
\newcommand{\R}{{\Bbb R}}
\newcommand{\Z}{{\Bbb Z}}
\newcommand{\cE}{{\cal E}}
\newcommand{\cF}{{\cal F}}
\newcommand{\cH}{{\cal H}}
\newcommand{\cL}{{\cal L}}
\newcommand{\cP}{{\cal P}}
\newcommand{\bone}{{\mathbf 1}}
\newcommand{\bfa}{\boldsymbol{a}}
\newcommand{\bfc}{\boldsymbol{c}}
\newcommand{\bfr}{\boldsymbol{r}}
\newcommand{\la}{\langle}\newcommand{\ra}{\rangle}
\newcommand{\wg}{\wedge}
\newcommand{\op}{\mathrm{op}}
\newcommand{\HS}{\mathrm{HS}}

\DeclareMathOperator{\tr}{tr}
\DeclareMathOperator{\rank}{rank}
\newcommand{\lsup}{\mathop{\overline{\lim}}}
\allowdisplaybreaks[3]

\title[Energy measures on Sierpinski gasket type fractals]{Some properties of energy measures\\ on Sierpinski gasket type fractals}

\author[M. Hino]{Masanori Hino\thanks{This research was partially supported by JSPS KAKENHI Grant Number 24540170.}\bigskip\\
{\small Dedicated to Professor Shinzo Watanabe on the occasion of his 80th birthday}}

\begin{document}

\maketitle

\begin{abstract}
We confirm, in a more general framework, a part of the conjecture posed by R.~Bell, C.-W.~Ho, and R.~S.~Strichartz~[Energy measures of harmonic functions on the Sierpi\'nski gasket, Indiana Univ. Math. J. \textbf{63} (2014), 831--868] on the distribution of energy measures for the canonical Dirichlet form on the two-dimensional standard Sierpinski gasket.
\end{abstract}



\section{Introduction}
Energy measures associated with strong local regular Dirichlet forms describe certain local structures of Dirichlet forms. For the standard energy form on a Euclidean space, the energy measure of a function $f$ is given explicitly by $|\nabla f(x)|^2\,dx$.
On the other hand, for canonical Dirichlet forms on fractals, energy measures do not usually have simple expressions and it seems a difficult problem to know how they are distributed in the state space. 
For example, energy measures are singular with respect to self-similar measures for self-similar Dirichlet forms on most self-similar fractals~\cite{Hi05,HN06,Ku89,BST99}.
Recently, Bell, Ho, and Strichartz~\cite{BHS} studied several properties of energy measures associated with the canonical Dirichlet form on the two-dimensional standard Sierpinski gasket.
In particular, they introduced non-negative coefficients $\{b_j^{(w)}\}_{j=1}^3$ for describing the distribution of the energy measures of harmonic functions on the cells corresponding to each word $w$ (see Section~2 for details), and posed conjectures about properties of the limiting behavior of $\{b_j^{(w)}\}_{j=1}^3$ as $w$ converges to an infinite sequence of words (\cite[Conjectures~7.1 and 7.2]{BHS}).
In this paper, we confirm a part of the conjecture in a slightly generalized setting.
The proof suggests that the conjectured properties depend strongly on the fractals under consideration having three vertices.
Our approach is more straightforward than the original one~\cite{BHS}: we use only primitive linear operators for the analysis and utilize some results on limits of random matrices (cf.~\cite{BL,Hi08,Hi10,Hi13,Ku89}).

The remainder of this paper is organized as follows. In Section~2, we provide a framework for Dirichlet forms on self-similar sets and give some preliminary results. The conjectures in \cite{BHS} are also stated. 
In Section~3, we prove the main theorem. Section~4 provides some concluding remarks and discussions.
\section{Framework and preliminaries}
We first introduce a class of self-similar sets and the Dirichlet forms defined on them, following~\cite{Ki}.
Let $K$ be a compact, connected, and metrizable space.
Let $\{\psi_i\}_{i\in S}$ be a family of continuous injective mappings from $K$ to itself having a finite index set $S$ with $\#S\ge 2$.
Denote $S^\N$ by $\Sg$ and each element of $\Sg$ by $\om_1\om_2\om_3\cdots$ with $\om_n\in S$ for every $n\in\N$.
For $i\in S$, a shift operator $\sg_i\colon \Sg\to \Sg$ is defined by $\sg_i(\om_1\om_2\cdots)=i\om_1\om_2\cdots$.
We assume that there exists a continuous surjective map $\pi\colon \Sg\to K$ such that $\psi_i\circ \pi=\pi\circ\sg_i$ for each $i\in S$.
The triplet $(K,S,\{\psi_i\}_{i\in S})$ is then called a self-similar structure.
Define $W_m=S^m$ for $m\in\Z_+:=\N\cup\{0\}$ and denote $\bigcup_{m\in\Z_+}W_m$ by $W_*$. For $w=w_1w_2\cdots w_m\in W_m\subset W_*$, we define $\psi_w=\psi_{w_1}\circ\psi_{w_2}\circ\cdots\circ\psi_{w_m}$ and $K_w=\psi_w(K)$.
For $w=w_1\cdots w_m\in W_m$ and $w'=w'_1\cdots w'_n\in W_n$, $ww'$ denotes $w_1\cdots w_m w'_1\cdots w'_n\in W_{m+n}$.
We set 
\[
\cP=\bigcup_{m=1}^\infty \sg^m\left(\pi^{-1}\left(\bigcup_{i,j\in S,\,i\ne j}(K_i\cap K_j)\right)\right)
\quad\mbox{and}\quad
V_0=\pi(\cP),
\]
where $\sg^m\colon \Sg\to\Sg$ is defined by $\sg^m(\om_1\om_2\cdots)=\om_{m+1}\om_{m+2}\cdots$.
We assume that $\cP$ is a finite set. In such a case, $(K,S,\{\psi_i\}_{i\in S})$ is called post-critically finite.
Then, from \cite[Lemma~1.3.14]{Ki}, each $\psi_i$ has a unique fixed point $\pi(iii\cdots)$.

For a finite set $V$, let $l(V)$ denote the space of all real-valued functions on $V$. A canonical inner product $(\cdot,\cdot)_{l(V)}$ on $l(V)$ is defined by $(u,v)_{l(V)}=\sum_{p\in V}u(p)v(p)$.
The associated norm is denoted by $|\cdot|_{l(V)}$.
Let $D=(D_{pq})_{p,q\in V_0}$ be a symmetric linear operator on $l(V_0)$ with the following properties.
\begin{enumerate}[(D1)]
\item $D$ is non-positive definite;
\item $Du=0$ if and only if $u$ is constant on $V_0$;
\item $D_{pq}\ge0$ for all $p,q\in V_0$ with $p\ne q$.
\end{enumerate}
Define $\cE^{(0)}(u,v)=(-Du,v)_{l(V_0)}$ for $u,v\in l(V_0)$.
For $m\in\N$, let $V_m=\break\bigcup_{w\in W_m}\psi_w(V_0)$. 
For $\bfr=\{r_i\}_{i\in S}$ with $r_i>0$ for all $i\in S$, define a bilinear form $\cE^{(m)}$ on $l(V_m)$ by
\[
  \cE^{(m)}(u,v)=\sum_{w\in W_m}\frac1{r_w}\cE^{(0)}(u\circ\psi_w|_{V_0},v\circ\psi_w|_{V_0}),\qquad
  u,v\in l(V_m),
\]
where $r_w=\prod_{i=1}^m r_{w_i}$ for $w=w_1w_2\cdots w_m\in W_m$.
The pair $(D,\bfr)$ is called a harmonic structure if 
\[
\cE^{(0)}(u,u)=\inf\{\cE^{(1)}(v,v)\mid v\in l(V_1),\ v|_{V_0}=u\}
\quad\text{for every $u\in l(V_0)$}.
\] 
Then, for $m\ge0$, the identity 
\[
\cE^{(m)}(u,u)=\inf\{\cE^{(m+1)}(v,v)\mid v\in l(V_{m+1}),\ v|_{V_m}=u\}
\]
holds for every $u\in l(V_{m})$.
If, moreover, $0<r_i<1$ for all $i\in S$, the harmonic structure is called regular.
Henceforth, we assume that a regular harmonic structure $(D,\bfr)$ is given.
Let $\mu$ be a finite Borel measure on $K$ with full support.
Then 
\[
\cE(u,v)= \lim_{m\to\infty}\cE^{(m)}(u|_{V_m},v|_{V_m}),\quad u,v\in\cF
\]
with
\[
\cF=\left\{u\in C(K)\subset L^2(K,\mu)\;\vrule\;
\sup_{m\in\Z_+}\cE^{(m)}(u|_{V_m},u|_{V_m})<\infty\right\}
\]
defines a bilinear form $(\cE,\cF)$ on $L^2(K,\mu)$ associated with $(D,\bfr)$. Under a mild condition on $\mu$, $(\cE,\cF)$ becomes a strong local regular Dirichlet form on $L^2(K,\mu)$ (self-similar measures are adequate, for example; see \cite[Theorem~3.4.6]{Ki}).
We will assume such a $\mu$ is chosen; the choice is not important for subsequent arguments.
We always take continuous functions as $\mu$-versions of elements of $\cF$.
\begin{example}\label{ex:1}
\begin{figure}\centering
\includegraphics[scale=1.2]{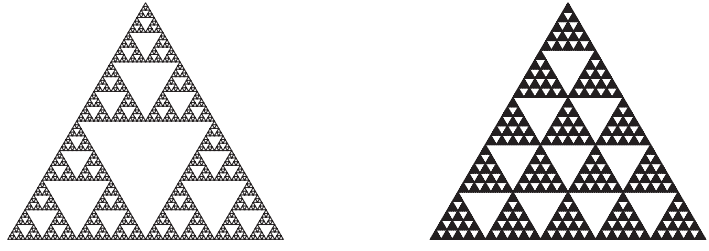}
\caption{Two-dimensional level $l$ Sierpinski gaskets $SG_l$ ($l=2,5$)}
\label{fig:1}
\end{figure}
Typical examples are two-dimensional level $l$ Sierpinski gaskets $SG_l$ for $l\ge2$, which are realized by compact subsets of $\R^2$ (see \Fig{1}). 
They are constructed by $l(l+1)/2$ contraction mappings $\psi_i$ defined as $\psi_i(z)=l^{-1}z+b_{l,i}$ with suitable $b_{l,i}\in \R^2$ and characterized by nonempty compact subsets satisfying $SG_l=\bigcup_{i=1}^{l(l+1)/2}\psi_i(SG_l)$.
We call $SG_2$ the two-dimensional standard Sierpinski gasket.
The set $V_0$ consists of the three vertices $p_1$, $p_2$, and $p_3$ of the largest triangle in $SG_l$.
We renumber $\{\psi_i\}_{i=1}^{l(l+1)/2}$ so that $\psi_i$ denotes the contraction mapping whose fixed point is $p_i$ for $i=1,2,3$ and define the matrix $D$ by 
\begin{equation*}
D=(D_{p_ip_j})_{i,j=1}^3=\begin{bmatrix}-2&1&1\\1&-2&1\\1&1&-2\end{bmatrix}.
\end{equation*}
Then, there exists a unique number $r$ such that $0<r<1$ and $(D,\bfr)$ is a regular harmonic structure with $\bfr=\{r,\dots,r\}$.
This example satisfies the conditions (A1), (A3), and (A4) that are stated later.
If we take the normalized Hausdorff measure as $\mu$, the diffusion process associated with the Dirichlet form as stated above is regarded as the Brownian motion on $SG_l$.
\end{example}
We now resume our discussion of the general situation.
For each $x\in l(V_0)$, there exists a unique function $h\in\cF$ such that $h|_{V_0}=x$ and $h$ attains the infimum of $\{\cE(g,g)\mid g\in\cF\mbox{ and } g|_{V_0}=x\}$.
Such a function $h$ is called a harmonic function and the totality of $h$ is denoted by $\cH$.
The map $\iota\colon l(V_0)\ni x\mapsto h\in \cH$ is linear, so we can identify $\cH$ with $l(V_0)$ by this map.
For $i\in S$, we define a linear operator $A_i\colon l(V_0)\to l(V_0)$ by $(A_i x)(p)=(\iota(x))(\psi_i(p))$ for $x\in l(V_0)$ and $p\in V_0$.
For $w=w_1w_2\cdots w_m\in W_m$, we set $A_w=A_{w_m}A_{w_{m-1}}\cdots A_{w_1}$.
With an abuse of notation, $D$ and $A_w$ can also be considered as linear maps from $\cH$ to $l(V_0)$ by identifying $\cH$ with $l(V_0)$.
Let $\bone\in l(V_0)$ denote a constant function on $V_0$ with value $1$.
Set 
\[
\tilde l(V_0)=\{x\in l(V_0)\mid (x,\bone)_{l(V_0)}=0\}
\]
and let $P\colon l(V_0)\to l(V_0)$ be the orthogonal projection onto $\tilde l(V_0)$.
\begin{lem}\label{lem:D}
The range of $D$ is $\tilde l(V_0)$.
\end{lem}
\begin{proof}
For $x\in l(V_0)$,
\[
(Dx,\bone)_{l(V_0)}=(x,D\bone)_{l(V_0)}=0.
\]
Therefore, $D(l(V_0))\subset \tilde l(V_0)$.
Since the dimensions of $D(l(V_0))$ and $\tilde l(V_0)$ are both $\# V_0-1$, we obtain the result.
\end{proof}

For $f\in\cF$, $\nu_f$ denotes the energy measure of $f$ (cf.~\cite{FOT}); in our situation, $\nu_f$ is the unique finite Borel measure on $K$ such that 
\[
    \int_K  g\, d\nu_{f}= 2\cE(f,fg)- \cE(f^2,g)
    \quad \text{for all }g\in \cF.
\]  
In particular, the energy measures of the constant functions are the zero measure in our framework.
From the general theory, it is known that every energy measure for strong local regular Dirichlet forms does not have a point mass (from the energy image density property; see, e.g., \cite[Theorem~4.3.8]{CF} or  \cite[Theorem~I.7.1.1]{BH} for the proof).
Therefore, energy measures $\nu_f$ do not have a mass on the countable set $V_*:=\bigcup_{m\in\Z_+}V_m$ for any $f\in\cF$.
A concrete expression for $\nu_f$ with $f\in\cH$ can be provided as follows.
\begin{lem}[cf.~{\cite[Lemma~4]{HN06}}]\label{lem:b}
For any $f\in\cH$ and $w\in W_*$,
\begin{equation}\label{eq:nu}
\nu_f(K_w)=-\frac2{r_w}\,^t(A_w f)D(A_w f).
\end{equation}
\end{lem}
We remark that the right-hand side of \Eq{nu} is also described as $(2/{r_w})\cE(f\circ\psi_w,f\circ\psi_w)$. This expression justifies an intuitive meaning of $\nu_f(K_w)$ as a ``local energy of $f$ on $K_w$.''\footnote{Compare this expression also with the following classical situation: for a Dirichlet form $(Q,H^1(\R^d))$ on $L^2(\R^d,dx)$, where $Q(f,g)=(1/2)\int_{\R^d}(\nabla f,\nabla g)_{\R^d}\,dx$ for $f,g\in H^1(\R^d)$, the energy measure of $f\in H^1(\R^d)$ is given by $|\nabla f|_{\R^d}^2\,dx$.}

The mutual energy measure $\nu_{f,g}$ for $f$ and $g$ in $\cF$ is a signed measure on $K$ defined by $\nu_{f,g}=(\nu_{f+g}-\nu_f-\nu_g)/2$.
For every Borel subset $B$ of $K$, the inequality
\begin{equation}\label{eq:CS}
|\nu_{f,g}(B)|^2\le \nu_f(B)\nu_g(B)
\end{equation}
holds.

We assume the following condition.
\begin{itemize}
\item[(A1)]
Each point of $V_0$ is a fixed point of some $\psi_j$.
More precisely, there exists a subset $S_0$ of $S$ such that $\#S_0=\#V_0$ and, for each $p\in V_0$, there exists $j\in S_0$ such that $\psi_j(p)=p$.
Furthermore, $K\setminus\{p\}$ is a connected set for every $p\in V_0$.
\end{itemize}
Under this condition, we have the following two lemmas.
\begin{lem}[cf.~{\cite[Theorem~A.1.2]{Ki}} and {\cite[Lemma~5]{HN06}}]\label{lem:A}
  Let $j\in S_0$.
  Take $p\in V_0$ such that $\psi_j(p)=p$ and
  let $u_j$ denote the column vector $(D_{qp})_{q\in V_0}$. Then
  \begin{enumerate}
  \item    $r_j$ is a simple eigenvalue of $A_j$ and $\,^t\!A_j$. Moreover, the modulus of all the eigenvalues of $A_j$ and $\,^t\!A_j$ other than $1$ and $r_j$ are less than $r_j$.
  \item The vector $u_j$ belongs to $\tilde l(V_0)$ and is an eigenvector of $\,^t\!A_j$ with respect to the eigenvalue $r_j$.
  \item There is a unique eigenvector $v_j$ of $A_j$ with respect to the eigenvalue $r_j$ such that $(u_j,v_j)_{l(V_0)}=1$. Moreover, every component of $v_j$ is non-negative.
\end{enumerate}
\end{lem}
\begin{lem}[cf.~{\cite[Lemma~6]{HN06}}]\label{lem:a}
For $j\in S_0$ and $x\in l(V_0)$, 
\[
\lim_{n\to\infty} r_j^{-n}P A_j^n x=(u_j,x)_{l(V_0)}Pv_j.
\]
\end{lem}
Let $h_1,\dots,h_N$ $(N\in\N)$ be a finite number of harmonic functions such that $\cH$ is spanned by $h_1,\dots,h_N$ and constant functions.
We denote $\sum_{k=1}^N \nu_{h_k}$ by $\nu$.
\begin{lem}\label{lem:abs}
For every $f\in\cH$, $\nu_f$ is absolutely continuous with respect to $\nu$.
\end{lem}
\begin{proof}
For some $\alpha_i\in\R$ $(i=1,\dots,N)$ and $\beta\in\R$, $f$ can be written as $f=\sum_{i=1}^N \alpha_i h_i+\beta$.
Then, $\nu_f=\sum_{i,j=1}^N \alpha_i\alpha_j\nu_{h_i,h_j}$.
By \Eq{CS}, this is absolutely continuous with respect to $\nu$.
\end{proof}
We further assume the following condition.
\begin{itemize}
\item[(A2)] For every $i\in S$, $A_i$ is invertible.
\end{itemize}
For $SG_l$ with the canonical harmonic structure in \Exam{1}, (A2) has been confirmed for $l\le 50$ by numerical computation (cf.~\cite[p.~297]{Hi10}).
It is conjectured that it is true for all $l\ge2$.
\begin{lem}\label{lem:easy}
For $j\in S_0$ and $w\in W_*$, $^{t\!}A_w u_j$ belongs to $\tilde l(V_0)$ and is nonzero. 
\end{lem}
\begin{proof}
This is clear because of the identity $A_w\bone=\bone$ and Condition (A2).
\end{proof}

For $j\in S_0$ and $w\in W_*$, let 
\[
a_j^{(w)}=\sum_{k=1}^N (u_j,A_w h_k)_{l(V_0)}^2.
\]
From \Lem{easy}, $a_j^{(w)}$ is strictly positive due to the choice of $\{h_k\}_{k=1}^N$.

We have some explicit information on the Radon--Nikodym derivative $d\nu_f/d\nu$ for harmonic functions $f$.
\begin{lem}\label{lem:jjj}
For any $f\in \cH$, $w\in W_*$, and $j\in S_0$, 
\[
\lim_{n\to\infty}\frac{\nu_f(K_{wj^n})}{\nu(K_{wj^n})}=\frac{(u_j,A_w f)_{l(V_0)}^2}{a_j^{(w)}},
\]
where $K_{wj^n}$ denotes $K_{w\underbrace{\scriptstyle j\dots j}_{n}}$.
\end{lem}
We denote this limit by $\frac{d\nu_f}{d\nu}({wj^\infty})$. 
Although this precise notation may look like $\frac{d\nu_f}{d\nu}(\pi({wj^\infty}))$, it is not clear whether the relation $\pi({wj^\infty})=\pi({\tilde w\tilde j^\infty})$ implies the identity $\frac{d\nu_f}{d\nu}({wj^\infty})=\frac{d\nu_f}{d\nu}({\tilde w\tilde j^\infty})$.
\begin{proof}[Proof of \Lem{jjj}]
From \Lem{b}, \Lem{a} and the identity $D={}^t\!P DP$,
\begin{align*}
\frac{\nu_f(K_{wj^n})}{\nu(K_{wj^n})}
&=\frac{^t(PA_j^n A_w f)D(PA_j^n A_w f)}{\sum_{k=1}^N {}^t(PA_j^nA_w h_k)D(PA_j^nA_w h_k)}\\
&\xrightarrow{n\to\infty} \frac{(u_j,A_w f)_{l(V_0)}^2\,^t(Pv_j)D(Pv_j)}{\sum_{k=1}^N (u_j,A_w h_k)_{l(V_0)}^2\,^t(Pv_j)D(Pv_j)}
=\frac{(u_j,A_w f)_{l(V_0)}^2}{a_j^{(w)}}.\tag*{\qedhere}
\end{align*}
\end{proof}
We now consider the following rather restrictive condition.
\begin{itemize}
\item[(A3)] $D^2=-\gm D$ for some $\gm>0$.
In other words, all the eigenvalues of $D$ are either $0$ or $-\gm$.
\end{itemize}
In \Exam{1}, (A3) holds with $\gm=3$.

The following proposition was proved in \cite[Theorem~6.1]{BHS} for the case of the standard Dirichlet form on $SG_2$; however, that proof was different from the one presented here.
\begin{prop}
Assume the conditions \rom{(A1)}--\rom{(A3)}.
For each $w\in W_*$, there exists a set of positive numbers $\{b_j^{(w)}\}_{j\in S_0}$ such that
\begin{equation}\label{eq:bjw}
\frac{\nu_f(K_w)}{\nu(K_w)}=\sum_{j\in S_0}b_j^{(w)}\frac{d\nu_f}{d\nu}(wj^\infty),
\qquad f\in \cH.
\end{equation}
If $\# V_0\ge3$, the $\{b_j^{(w)}\}_{j\in S_0}$ are uniquely determined.
\end{prop}
\begin{proof}
Since $^tD D=D^2=-\gm D$,
\begin{align*}
\frac{\nu_f(K_w)}{\nu(K_w)}
&=\frac{^t(A_w f)D(A_w f)}{\sum_{k=1}^N {}^t(A_w h_k)D(A_w h_k)}\\
&=\frac{-\gm^{-1}\cdot{}^t(DA_w f)(DA_w f)}{-\gm^{-1}\sum_{k=1}^N {}^t(DA_w h_k)(DA_w h_k)}\\
&=\frac{\sum_{j\in S_0}(u_j,A_w f)_{l(V_0)}^2}{\sum_{k=1}^N\sum_{i\in S_0}(u_i,A_w h_k)_{l(V_0)}^2}\\
&=\sum_{j\in S_0}\frac{a_j^{(w)}}{\sum_{i\in S_0} a_i^{(w)}}\frac{d\nu_f}{d\nu}(wj^\infty).
\end{align*}
Therefore, \Eq{bjw} holds by letting
\begin{equation}\label{eq:bjex}
b_j^{(w)}=\frac{a_j^{(w)}}{\sum_{i\in S_0} a_i^{(w)}}.
\end{equation}

To prove the uniqueness of $\{b_j^{(w)}\}_{j\in S_0}$, it suffices to prove that
\begin{equation}\label{eq:unique}
\sum_{j\in S_0}\beta_j(u_j,A_w f)_{l(V_0)}^2=0\quad\mbox{for all $f\in \cH$ }
\end{equation}
implies $\beta_j=0$ for all $j\in S_0$.
Let $j$ and $k$ be distinct elements of $S_0$.
Denote the fixed points of $\psi_j$ and $\psi_k$ by $p_j$ and $p_k$, respectively.
From \Lem{D}, there exists an $x\in l(V_0)$ such that $(Dx)(p_j)=1$, $(Dx)(p_k)=-1$, and $(Dx)(p)=0$ for $p\in V_0\setminus\{p_j,p_k\}$; in other words, $(u_j,x)_{l(V_0)}=1$, $(u_k,x)_{l(V_0)}=-1$, and $(u_i,x)_{l(V_0)}=0$ for $i\in S_0\setminus\{j,k\}$. 
From the surjectiveness of $A_w$, there exists an $f\in \cH\simeq l(V_0)$ such that $A_w f=x$.
Then, from \Eq{unique}, $\beta_j+\beta_k=0$.
This relation implies that $\beta_j=0$ for all $j\in S_0$ because $\#S_0=\#V_0\ge3$.
\end{proof}
From \Eq{bjex}, we have the identity
\begin{equation}\label{eq:sum1}
\sum_{j\in S_0}b_j^{(w)}=1.
\end{equation}
The coefficients $\{b_j^{(w)}\}_{j\in S_0}$ provide some information on the distribution of the energy measures of harmonic functions.
In a typical example, $\{b_j^{(w)}\}_{j\in S_0}$ describes the skewness of $\nu$ on the cell $K_{wj}$ relative to $K_w$ as follows, which is due to Bell, Ho, and Strichartz~\cite{BHS}.
\begin{thm}[cf.\ {\cite[Theorem~6.3]{BHS}}]\label{th:b}
We consider the standard Dirichlet form on $SG_2$ given in \Exam{1} and  write $S=S_0=\{1,2,3\}$.  
As a choice of $\{h_i\}_{i=1}^N\subset\cH$, let $N=2$ and take a pair $h_1,h_2\in\cH$ so that $\cE(h_i,h_j)=\dl_{ij}/4$ for any $i,j\in\{1,2\}$, where $\dl_{ij}$ represents the Kronecker delta.
Accordingly, $\nu=\nu_{h_1}+\nu_{h_2}$.
Then, the identity
\[
\frac15\left(b_j^{(w)}-\frac13\right)=\frac14\left(\frac{\nu(K_{wj})}{\nu(K_w)}-\frac13\right)
\]
holds for any $w\in W_*$ and $j\in S_0$.
\end{thm}
It is easy to see that the measure $\nu$ given above is a probability measure on $K$ and is independent of the choice of $h_1$ and $h_2$.
For $SG_l$ with $l\ge3$, such a clear interpretation of $\{b_j^{(w)}\}_{j\in S_0}$ as in \Thm{b} seems difficult to obtain.

Let $\lm$ be a Borel probability measure on $\Sg=S^\N$, defined as the infinite product of the uniform probability measure $(\#S)^{-1}\sum_{j\in S}\dl_j$ on $S$. 
For $\om=\om_1\om_2\cdots\in \Sg$ and $m\in\N$, $[\om]_m$ denotes $\om_1\om_2\cdots\om_m\in W_m$.
For $m\in \Z_+$, let $\xi_m$ denote the image measure of $\lm$ by the map $\Sg\ni \om\mapsto \{b_j^{([\om]_m)}\}_{j\in S}\in \R^{\#S}$.
 Bell, Ho, and Strichartz~\cite{BHS} discuss some properties of $\xi_m$ for the canonical Dirichlet form on $SG_2$ and posed conjectures, which we call \Conj{BHS2} below.

 Until the end of this section, we consider the Dirichlet form for two-dimensional standard Sierpinski gasket $SG_2=:K$ given in \Exam{1} and take the measure $\nu$ as in \Thm{b}.
 We write $S=S_0=\{1,2,3\}$ and 
 \begin{align*}
 \D:={}&\Biggl\{(b_1,b_2,b_3)\in\R^3\;\vrule\; \sum_{j=1}^3 b_j=1\text{ and }\sum_{j=1}^3\left(b_j-\frac13\right)^2<\frac16\Biggr\}\\
 \Biggl(={}&\Biggl\{(b_1,b_2,b_3)\in\R^3\;\vrule\; \sum_{j=1}^3 b_j=1\text{ and }\sum_{j=1}^3 b_j^2<\frac12\Biggr\}\Biggl).
 \end{align*}
 \begin{thm}[cf.~{\cite[Theorem~6.5]{BHS}}]\label{th:BHS1}
 For all $w\in W_*$, $(b_1^{(w)},b_2^{(w)},b_3^{(w)})$ belongs to $\D$; that is,
 \begin{equation}\label{eq:less}
 \sum_{j=1}^3\left(b_j^{(w)}-\frac13\right)^2<\frac16.
 \end{equation}
 This inequality is sharp.
 In particular, $\xi_m$ concentrates on $\D$ for all $m$.
 \end{thm}
We note that \Eq{less} can be rewritten as
\[
 \sum_{j=1}^3 (b_j^{(w)})^2<\frac12
\]
because of \Eq{sum1}.

\begin{conj}[cf.~{\cite[Conjectures~7.1 and 7.2]{BHS}}]\label{conj:BHS2}
Let $(r,\theta)$ be polar coordinates for the disk $\D$ with center $\bfc=(1/3,1/3,1/3)$.
More specifically, 
\begin{align*}
r(z)&=|z-\bfc|_{\R^3},\\
\theta(z)&=\mathop{\rm Arg}\left((z-\bfc,\bfa_1)_{\R^3}+\sqrt{-1}(z-\bfc,\bfa_2)_{\R^3}\right)\in(-\pi,\pi]
\end{align*}
with $\bfa_1=(1/\sqrt2,-1/\sqrt2,0)$, $\bfa_2=(-1/\sqrt6,-1/\sqrt6,2/\sqrt6)$, where $(\cdot,\cdot)_{\R^3}$ and $|\cdot|_{\R^3}$ denote the standard inner product and norm on $\R^3$, respectively.
For $m\in \N$, let $P_m$ and $Q_m$ denote the image measures of $\xi_m$ by the mappings $\theta(\cdot)$ and $r(\cdot)$, respectively.
Then:
\begin{enumerate}
\item $P_m$ converges weakly to an absolutely continuous measure on $(-\pi,\pi]$ as $m\to\infty$;
\item $Q_m$ converges weakly to the delta measure at $1/\sqrt6$ as $m\to\infty$.
\end{enumerate}
\end{conj}
Bell, Ho, and Strichartz~\cite{BHS} also conjectured the invariance of the limit of $P_m$ under some rational maps, but we skip the details because we do not discuss such kind of property in this paper.

In the next section we prove \Thm{BHS1} and confirm \Conj{BHS2}(2) in a slightly more general situation.

\section{Main results}
We keep the notation used in the previous section and always assume conditions (A1)--(A3).

Fix $w\in W_*$.
For $j\in S_0$, let $z_j={}^t\!A_w u_j$.
Note that $z_j\in \tilde l(V_0)$ and $z_j\ne0$ from \Lem{easy}.
Also, since $\sum_{j\in S_0}u_j=0$ from $D\bone=0$, we have
\begin{equation}\label{eq:zz}
\sum_{j\in S_0}z_j=0.
\end{equation}
For $x,y\in \tilde l(V_0)\subset l(V_0)$, we define
\[
\la x,y\ra=\sum_{k=1}^N (x,h_k)_{l(V_0)}(y,h_k)_{l(V_0)}\quad\mbox{and}\quad\|x\|=\la x,x\ra^{1/2}.
\]
Then, $\la\cdot,\cdot\ra$ is an inner product on $\tilde l(V_0)$ and the identity $\|z_j\|^2=a_j^{(w)}$ holds.
We remark that there exists a positive definite symmetric operator $H$ on $\tilde l(V_0)$ such that $\la x,y\ra=(Hx,Hy)_{l(V_0)}$ for all $x,y\in \tilde l(V_0)$.

Fix an arbitrary $k\in S_0$ and let $S_0'=S_0\setminus\{k\}$. Then we have the following lemma.
\begin{lem}\label{lem:bj}
The following identity holds:
\begin{equation}\label{eq:bj}
\sum_{j\in S_0}(b_j^{(w)})^2
=\left(2+\frac{2\sum_{i,j\in S_0',\,i\ne j}\|z_i\|^2\|z_j\|^2-\left(\sum_{i,j\in S_0',\,i\ne j}\la z_i,z_j\ra\right)^2}{\sum_{j\in S_0}\|z_j\|^4}\right)^{-1}.
\end{equation}
\end{lem}
\begin{proof}
From \Eq{bjex},
\begin{align}\label{eq:sumb2}
\sum_{j\in S_0}(b_j^{(w)})^2&=\frac{\sum_{j\in S_0}(a_j^{(w)})^2}{\left(\sum_{j\in S_0}a_j^{(w)}\right)^2}
=\frac{\sum_{j\in S_0}\|z_j\|^4}{\left(\sum_{j\in S_0}\|z_j\|^2\right)^2}\nonumber\\
&=\left(2+\frac{\left(\sum_{j\in S_0}\|z_j\|^2\right)^2-2\sum_{j\in S_0}\|z_j\|^4}{\sum_{j\in S_0}\|z_j\|^4}\right)^{-1}.
\end{align}
Using the identity $z_k=-\sum_{j\in S_0'}z_j$ from \Eq{zz}, we have
\begin{align}
\sum_{j\in S_0}\|z_j\|^4
&=\sum_{j\in S_0'}\|z_j\|^4+\left\|\sum_{j\in S_0'}z_j\right\|^4\nonumber\\
&=\sum_{j\in S_0'}\|z_j\|^4+\left(\sum_{j\in S_0'} \|z_j\|^2+\sum_{i,j\in S_0',\,i\ne j}\la z_i,z_j\ra\right)^2\nonumber\\
&=2\sum_{j\in S_0'}\|z_j\|^4+\sum_{i,j\in S_0',\,i\ne j}\|z_i\|^2\|z_j\|^2+2\sum_{j\in S_0'} \|z_j\|^2\sum_{i,j\in S_0',\,i\ne j}\la z_i,z_j\ra\nonumber\\
&\quad+\left(\sum_{i,j\in S_0',\,i\ne j}\la z_i,z_j\ra\right)^2
\label{eq:bj1}
\end{align}
and
\begin{align}
&\left(\sum_{j\in S_0}\|z_j\|^2\right)^2-\sum_{j\in S_0}\|z_j\|^4\nonumber\\
&=\sum_{i,j\in S_0,\,i\ne j}\|z_i\|^2\|z_j\|^2\nonumber\\
&=\sum_{i,j\in S_0',\,i\ne j}\|z_i\|^2\|z_j\|^2+2\left\|\sum_{j\in S_0'} z_j\right\|^2\sum_{j\in S_0'} \|z_j\|^2\nonumber\\
&=\sum_{i,j\in S_0',\,i\ne j}\|z_i\|^2\|z_j\|^2
+2\left(\sum_{j\in S_0'} \|z_j\|^2\right)^2+2\sum_{i,j\in S_0',\,i\ne j}\la z_i,z_j\ra\sum_{j\in S_0'}\|z_j\|^2\nonumber\\
&=3\sum_{i,j\in S_0',\,i\ne j}\|z_i\|^2\|z_j\|^2+2\sum_{j\in S_0'}\|z_j\|^4
+2\sum_{i,j\in S_0',\,i\ne j}\la z_i,z_j\ra\sum_{j\in S_0'}\|z_j\|^2.
\label{eq:bj2}
\end{align}
Then,
\begin{equation*}
\left(\sum_{j\in S_0}\|z_j\|^2\right)^2-2\sum_{j\in S_0}\|z_j\|^4=2\sum_{i,j\in S_0',\,i\ne j}\|z_i\|^2\|z_j\|^2-\left(\sum_{i,j\in S_0',\,i\ne j}\la z_i,z_j\ra\right)^2
\end{equation*}
by combining \Eq{bj1} and \Eq{bj2}.
This identity and \Eq{sumb2} imply \Eq{bj}.
\end{proof}
Lastly we consider the following condition.
\begin{itemize}
\item[(A4)]$\#V_0=3$.
\end{itemize}
The following extends \Thm{BHS1} (\cite[Theorem~6.5]{BHS}) to more general situations, and the proof is more straightforward.
\begin{thm}\label{th:bj3}
Under the conditions {\rm(A1)--(A4)}, 
\begin{equation}\label{eq:12}
\sum_{j\in S_0}(b_j^{(w)})^2<\frac12
\end{equation}
for all $w\in W_*$.
This inequality is sharp.
\end{thm}
\begin{proof}
Let $S'_0=\{1,2\}$. Then, \Eq{bj} can be rewritten as
\begin{equation}\label{eq:bj3}
\sum_{j\in S_0}(b_j^{(w)})^2
=\left(2+4\cdot\frac{\|z_{1}\|^2\|z_{2}\|^2-\la z_{1},z_{2}\ra^2}{\sum_{j\in S_0}\|z_j\|^4}\right)^{-1}.
\end{equation}
Moreover, the inequality $|\la z_{1},z_{2}\ra|\le \|z_{1}\|\|z_{2}\|$ holds with equality if and only if $z_{1}$ and $z_{2}$ are linearly dependent.
Since $u_1$ and $u_2$ are linearly independent by the property (D2) of $D$, the inequality is strict. 
Therefore, we obtain \Eq{12}.
The sharpness of this inequality is confirmed by \Thm{main} below, so we omit the proof here.
\end{proof}
\begin{remark}
As can be seen from the proof above, it seems difficult to obtain a good estimate of $\sum_{j\in S_0}(b_j^{(w)})^2$ if $\#V_0>3$.
Indeed, if $\#V_0=4$ and $S_0=\{1,2,3,4\}$, Eq.~\Eq{bj} is rewritten as
\[
\sum_{j\in S_0}(b_j^{(w)})^2
=\left(2+4\cdot\frac{I}{\sum_{j\in S_0}\|z_j\|^4}\right)^{-1}
\]
with
\begin{align*}
I=&\sum_{(i,j)\in\{(1,2),\, (2,3),\, (3,1)\}}(\|z_{i}\|^2\|z_{j}\|^2-\la z_{i},z_{j}\ra^2)\\
&-2\sum_{(i,j,k)\in\{(1,2,3),\, (2,3,1),\, (3,1,2)\}}\la z_i,z_j\ra\la z_j,z_k\ra.
\end{align*}
We may need other functionals to specify the range of $\{b_j^{(w)}\}_{j\in S_0}$ in such a case.
\end{remark}

Let $\tilde D$ denote the restriction of $D$ as a negative definite symmetric operator on $\tilde l(V_0)$.
For $w\in W_*$, let $\tilde A_w$ denote the restriction of $P A_w(= P A_wP)$ as a linear operator on $\tilde l(V_0)$.

For the statement of the main theorem, we recall the concept of strong irreducibility of random matrices.
\begin{definition}
 A set $\mathcal{T}$ of invertible linear operators on $\tilde l(V_0)$ is called strongly irreducible if there does not exist a finite family $L_1,\dots,L_k$ of proper linear subspaces of $\tilde l(V_0)$ such that
 $M(L_1\cup\dots\cup L_k)=L_1\cup\dots\cup L_k$ for all $M\in \mathcal{T}$. 
\end{definition}
\begin{example}
We again consider a canonical harmonic structure on the two-\hspace{0pt}dimensional level $l$ Sierpinski gasket in \Exam{1}. Further, we assume \rom{(A2)}. Then we can prove that $\{\tilde A_i\}_{i\in S}$ is strongly irreducible.
Indeed, from (A2) and the fact that the sequence $\left\{\left(|\det\tilde A_j|^{-1/2}\tilde A_j\right)^n\right\}_{n=1}^\infty$ is unbounded for $j\in S_0$, as shown in the proof of \Thm{main} below, it suffices to prove the following claim by \cite[Part~A, Chapter~II, Proposition~4.3]{BL}: 
\begin{equation}\label{eq:check}
\parbox{0.85\hsize}{For every $x\in \tilde l(V_0)\setminus\{0\}$, the set $\{\tilde A_i^n x\mid i\in S,\ n\in\Z_+\}$ has three elements $y_1,y_2,y_3$ such that $y_j$ and $y_k$ are pairwise linearly independent for $j\ne k$.}
\end{equation}
Let $S_0=\{1,2,3\}$. From the symmetry of the harmonic structure and\break\Lem{A}(1), $\tilde A_1$ has two different eigenvalues and the eigenvectors of $\tilde A_1$ (up to multiplicative constants) are $z_1:={}^t(2,-1,-1)$ and $z_2:={}^t(0,1,-1)$. 
The set of eigenvectors of $\tilde A_1^2$ is the same as that of $\tilde A_1$.
The same claims hold for $\tilde A_2$ with $z_1$ and $z_2$ replaced by ${}^t(-1,2,-1)$ and ${}^t(1,0,-1)$, respectively. 
 Now, let $x\in \tilde l(V_0)\setminus\{0\}$. If $x$ and $z_i$ are linearly dependent for $i=1$ or $2$, any two of $\{x, A_2x, A_2^2x\}$ are linearly independent. Otherwise, any two of $\{x, A_1x, A_1^2x\}$ are linearly independent. 
 Therefore, \Eq{check} follows.
\end{example}
We now resume our discussion of the general situation. The following is the main theorem of this paper.
\begin{thm}\label{th:main}
Assume the conditions {\rm(A1)--(A4)}.
Let $\kp$ be a Borel measure on $\Sg=S^\N$.
We further suppose \emph{either} of the following cases.
\begin{enumerate}[\rm(I)]
\item $\kp$ is an infinite product of a probability measure on $S$ with full support, and $\{\tilde A_i\}_{i\in S}$ is strongly irreducible. 
\item The image measure of $\kp$ by $\pi\colon \Sg\to K$ is absolutely continuous with respect to $\nu$.
\end{enumerate}
Then,
\begin{equation}\label{eq:main}
\lim_{n\to\infty}\sum_{j\in S_0}(b_j^{([\om]_n)})^2=\frac12 \quad\mbox{for $\kp$-a.e.\,$\om$.}
\end{equation}
In particular, the image measure of $\kp$ by the map $\Sg\ni \om\mapsto \{b_j^{([\om]_n)}\}_{j\in S_0}\in l(S_0)\cong l(V_0)$ converges weakly as $n\to\infty$ to a measure that concentrates on the set $\left\{x\in l(V_0)\;\vrule\; \sum_{p\in V_0}x(p)=1\text{ and }|x|_{l(V_0)}^2=1/2\right\}$.
\end{thm}
The result for Case (I) gives an affirmative answer to \Conj{BHS2}(2).
We remark that the strong irreducibility of $\{\tilde A_i\}_{i\in S}$ is not necessary in Case~(II).
\begin{proof}[Proof of \Thm{main}]
First, we note that $\tilde l(V_0)$ is two-dimensional because of (A4).
Using the same notation as in \Thm{bj3}, with $w=[\om]_n$ for $\om\in\Sg$ and $n\in\N$,
\begin{align*}
\|z_1\|^2\|z_2\|^2-\la z_1,z_2\ra^2
&=|Hz_1|_{l(V_0)}^2|Hz_2|_{l(V_0)}^2-(H z_1,Hz_2)_{l(V_0)}^2\\
&=|Hz_1\wg Hz_2|_{\bigwedge^2\tilde l(V_0)}^2\\
&=(\det H)^2(\det \tilde A_{[\om]_n})^2| u_1\wg  u_2|_{\bigwedge^2\tilde l(V_0)}^2.
\end{align*}
Moreover, since $u_1$ and $u_2$ are linearly independent, the map $B\mapsto (\|Bu_1\|^4+\|Bu_2\|^4)^{1/4}$ provides a norm on the space $\cL(\tilde l(V_0))$ of all linear operators on $\tilde l(V_0)$.
Therefore,
\[
\sum_{j\in S_0}\|z_j\|^4\ge \|\tilde A_{[\om]_n} u_1\|^4+ \|\tilde A_{[\om]_n} u_2\|^4\ge c\|\tilde A_{[\om]_n}\|_\op^4,
\]
where $\|\cdot\|_\op$ represents the operator norm on $\cL(\tilde l(V_0))$ and $c$ is a positive constant independent of $\om$ and $n$.
Then 
\[
0<\frac{\|z_1\|^2\|z_2\|^2-\la z_1,z_2\ra^2}{\sum_{j\in S_0}\|z_j\|^4}
\le \frac{(\det H)^2(\det \tilde A_{[\om]_n})^2 \left| u_1\wg  u_2\right|_{\bigwedge^2\tilde l(V_0)}^2}{c\|\tilde A_{[\om]_n}\|_\op^4}.
\]
By virtue of \Eq{bj3}, Eq.~\Eq{main} follows if we can prove that
\begin{equation}\label{eq:A}
\lim_{n\to\infty}\frac{(\det \tilde A_{[\om]_n})^2}{\|\tilde A_{[\om]_n}\|_\op^4}=0
\quad \mbox{for $\kp$-a.e.\,$\om$.}
\end{equation}

Suppose Case (I) and fix $j\in S_0$. Since $\tilde A_j$ is invertible from (A2) and has two eigenvalues with different moduli from \Lem{A}(1), the sequence\break $\left\{\left(|\det\tilde A_j|^{-1/2}\tilde A_j\right)^n\right\}_{n=1}^\infty$ is unbounded. 
Together with strong irreducibility of $\{A_i\}_{i\in S}$, Furstenberg's theorem (cf.~\cite[Part A, Chapter II, Theorems~4.1 and 3.6]{BL}) implies that for two elements $x_1,x_2\in \tilde l(V_0)$ which are linearly independent, 
\[
  \lim_{n\to\infty}\dl(\tilde A_{[\om]_n}x_1,\tilde A_{[\om]_n}x_2)=0
  \quad\text{for $\kp$-a.e.\,$\om$}.
\]
Here, $\dl(\cdot,\cdot)\in[0,1]$ denotes the angular distance, that is,
\[
 \dl(y_1,y_2)=\sqrt{1-\left(\frac{y_1}{|y_1|_{l(V_0)}},\frac{y_2}{|y_2|_{l(V_0)}}\right)_{l(V_0)}^2} \quad \text{for }y_1,y_2\in\tilde l(V_0)\setminus\{0\}.
\]
Let $x_1,x_2$ be an orthonormal basis of the inner product space $(\tilde l(V_0),\ (\cdot,\cdot)_{l(V_0)})$.
Then
\[
\dl(\tilde A_{[\om]_n}x_1,\tilde A_{[\om]_n}x_2)
=\frac{|\det \tilde A_{[\om]_n}|}{|\tilde A_{[\om]_n}x_1|_{l(V_0)}|\tilde A_{[\om]_n}x_2|_{l(V_0)}}
\ge\frac{|\det \tilde A_{[\om]_n}|}{\|\tilde A_{[\om]_n}\|_\op^2}.
\]
Thus we obtain \Eq{A}.

Next, suppose Case (II).
Again let $x_1,x_2$ be an orthonormal basis of $\tilde l(V_0)$.
Define $\hat h_i=\iota(x_i)$ for $i=1,2$ and $\hat\nu=\nu_{\hat h_1}+\nu_{\hat h_2}$.
Since the linear span of $\hat h_1$, $\hat h_2$, and constant functions is $\cH$, $\nu$ and $\hat \nu$ are mutually absolutely continuous from \Lem{abs}. Therefore, we may assume that the image measure of $\kp$ by the mapping $\pi$ coincides with $\hat\nu$ for proving \Eq{A}.\footnote{Since $\hat\nu(V_*)=0$, such a $\kp$ is uniquely identified. More specifically, for $m\in\N$ and $A\subset W_m$, $\kp(\{\om\in\Sg\mid[\om]_m\in A\})$ is given by $\sum_{k=1}^2 \sum_{w\in A}2r_w^{-1}\cE(\hat h_k\circ\psi_w,\hat h_k\circ\psi_w)$ from \Lem{b}.}
In our situation, the index (\cite[Definition~2.9]{Hi10}) of the Dirichlet form under consideration is $1$ (see \cite[Proposition~3.4]{Hi08} or \cite[Theorem~4.10]{Hi13} for the proof; see also \cite{Ku89}). 
In particular, we have
\[
\rank\left(\frac{d\nu_{\hat h_i,\hat h_j}}{d\hat\nu}\right)_{i,j=1}^2\le 1
\quad\mbox{$\hat\nu$-a.e.}\footnote{In fact, the equality holds from \cite[Proposition~2.11]{Hi10}.}
\]
On the other hand, for $i,j\in\{1,2\}$, 
\[
\frac{d\nu_{\hat h_i,\hat h_j}}{d\hat\nu}(\pi(\om))
=\lim_{n\to\infty}\frac{\nu_{\hat h_i,\hat h_j}(K_{[\om]_n})}{\hat\nu(K_{[\om]_n})}
=\lim_{n\to\infty}\frac{-\,{}^t\!x_i\,^t\!\tilde A_{[\om]_n}\tilde D\tilde A_{[\om]_n}x_j}{\tr\left(-\,{}^t\!\tilde A_{[\om]_n}\tilde D\tilde A_{[\om]_n}\right)}
\quad\mbox{for $\kp$-a.e.\,$\om$},
\]
where the first equality follows from the martingale convergence theorem and the second one is due to \Lem{b}.
Then, for $\kp$-a.e.\,$\om$,
\begin{align*}
0&=\lim_{n\to\infty}\det\left(\frac{-\,{}^t\!\tilde A_{[\om]_n}\tilde D\tilde A_{[\om]_n}}{\tr\left(-\,{}^t\!\tilde A_{[\om]_n}\tilde D\tilde A_{[\om]_n}\right)}\right)\\
&=\lim_{n\to\infty}\frac{\left(\det\tilde A_{[\om]_n}\right)^2\det(- \tilde D)}{\left\|\sqrt{ -\tilde D}\tilde A_{[\om]_n}\right\|_{\HS}^4}\\
&\ge\lsup_{n\to\infty}\frac{\left(\det\tilde A_{[\om]_n}\right)^2\det(- \tilde D)}{c'\|\tilde A_{[\om]_n}\|_{\op}^4},
\end{align*}
where $\|\cdot\|_{\HS}$ denotes the Hilbert--Schmidt norm on $\cL(\tilde l(V_0))$ and $c'$ is a positive constant depending only on $D$.
Thus, \Eq{A} holds.
 
The last claim of the theorem follows from the general fact that almost sure convergence implies convergence in law.
\end{proof}
\section{Concluding remarks}
We give some comments as concluding remarks.
\begin{enumerate}
\item As can be seen from the proof, Condition (A4) is crucial for \Thm{bj3} and thus for \Thm{main}. It may be an interesting problem to provide an appropriate formulation when $\# V_0>3$.
\item In both cases (I) and (II) in \Thm{main}, $\kp$ has no mass on $\pi^{-1}(V_*)$.
Therefore, the statements of \Thm{main} and \Conj{BHS2} can be rephrased in terms of a measure on $K$ instead of the measure $\kp$ on $\Sg$: that is, self-similar measures on $K$ in Case~(I) and $\nu$ in Case~(II), respectively.
\item In \Thm{main}, the measure $\kp$ of Case (I) and that of Case (II) are mutually singular in many cases (cf.~\cite[Theorem~2]{HN06}). Case~(II) looks like a more natural formulation in the sense that there is no need for the extra assumption of the strong irreducibility of $\{A_j\}_{j\in S}$ and because the concept of the index of Dirichlet forms, which also has probabilistic interpretations~\cite{Hi10,Hi13b}, appears naturally in the proof.
\item At the moment, there are no clues concerning \Conj{BHS2}(1). The distribution of $P_m$ with $m=13$ is given in the left-hand graph of \Fig{2}.
\begin{figure}\centering
\includegraphics[width=\hsize]{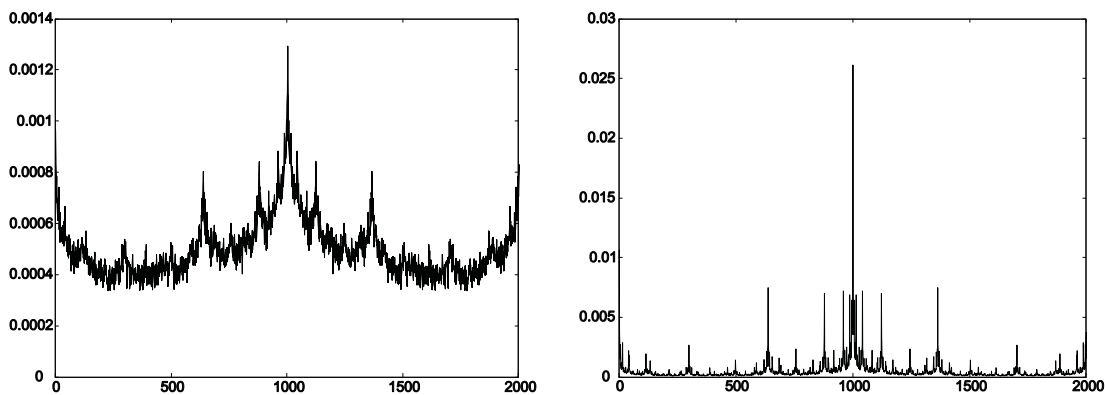}
\caption{Angular distributions}\label{fig:2}
\end{figure}
This figure shows the normalized histogram (2000 slices) of $P_{13}$ on $[-\pi/3,\pi/3]$ for the Dirichlet form on $SG_2$ assumed in \Conj{BHS2}; because of the symmetry it suffices to consider only this interval.\footnote{The histogram of $P_{13}$ with 100 slices is also provided in \cite[Figure~7.4]{BHS}; the figure there corresponds to the right half of ours. However, with only 100 slices the graph does not seem to reflect the irregular behavior of the distribution of $P_m$ very precisely.}
This figure supports the validity of \Conj{BHS2}(1).
On the other hand, the right-hand graph of \Fig{2} shows the distribution on $[-\pi/3,\pi/3]$ of the image measure of $\kp$ by the map $\Sg\ni \om\mapsto \theta(\{b_j^{([\om]_{m})}\}_{j=1}^3)\in (-\pi,\pi]$ with $m=13$, where $\kp$ is taken so that its image measure by the map $\pi\colon\Sg\to K$ is equal to the measure~$\nu$ given in \Thm{b}.  
Here, the interval $[-\pi/3,\pi/3]$ is again divided into 2000~slices.
The distribution looks very different and the possible limit measure as $m\to\infty$ might be singular with respect to the Lebesgue measure.
\end{enumerate}

\end{document}